\documentclass[a4paper,12pt]{article}

\usepackage[paper=a4paper, left=28mm, right=25mm, top=29mm, bottom=25mm]{geometry}
\usepackage{amsfonts}
\usepackage{amssymb}
\usepackage{latexsym}
\usepackage{amsmath}
\usepackage{amsthm}

\usepackage[usenames]{color}
\definecolor{darkred}{RGB}{139,0,0}
\definecolor{darkgreen}{RGB}{0,100,0}
\definecolor{darkmagenta}{RGB}{139,0,139}
\definecolor{darkpurple}{RGB}{110,0,180}
\definecolor{darkblue}{RGB}{40,0,200}
\definecolor{darkorange}{RGB}{255,140,0}

\newtheorem{lemma}{Lemma}
\newtheorem{theorem}{Theorem}
\newtheorem{corollary}{Corollary}
\newtheorem{remark}{Remark}
\newtheorem{iremark}[remark]{Important Remark}
\newtheorem{definition}{Definition}

\newcommand{\F}{\mathbb{F}}
\newcommand{\NN}{\mathbb{N}}
\newcommand{\ZZ}{\mathbb{Z}}
\newcommand{\cT}{\mathcal{T}}
\newcommand{\bsc}{{\boldsymbol{c}}}
\newcommand{\bsk}{{\boldsymbol{k}}}
\newcommand{\bsh}{{\boldsymbol{h}}}

\newcommand{\bsw}{{\boldsymbol{w}}}
\newcommand{\bsx}{{\boldsymbol{x}}}
\newcommand{\bsy}{{\boldsymbol{y}}}
\newcommand{\bsz}{{\boldsymbol{z}}}
\newcommand{\bszero}{{\boldsymbol{0}}}
\newcommand{\bsone}{{\boldsymbol{1}}}
\newcommand{\bsalpha}{{\boldsymbol{\alpha}}}
\newcommand{\bsgamma}{{\boldsymbol{\gamma}}}
\newcommand{\rd}{\,{\rm d}}
\newcommand{\uu}{\mathfrak{u}}
\newcommand{\cP}{\mathcal{P}}
\newcommand{\cQ}{\mathcal{Q}}
\newcommand{\icomp}{\mathtt{i}}

\title{Digital inversive vectors can achieve strong polynomial tractability for the weighted star discrepancy and for multivariate integration}

\author{Josef Dick, 
Domingo Gomez-Perez, \\
Friedrich Pillichshammer, 
Arne Winterhof\thanks{The research of J.\ Dick was supported under the Australian Research Councils Discovery Projects funding scheme (project number DP150101770). 
The research of D.\ Gomez-Perez is supported by the Ministerio de Economia y Competitividad research project MTM2014-55421-P.
The last two authors are supported by the
Austrian Science Fund (FWF): Projects F5509-N26 (Pillichshammer) and F5511-N26 (Winterhof), respectively, which are part of the
Special Research Program ``Quasi-Monte Carlo Methods: Theory and
Applications''.} 
}

\date{}

\begin{document}

\maketitle

\begin{center} 
In memory of Joseph Frederick Traub (1932--2015)\\ and \\ Oscar Moreno de Ayala (1946--2015)
\end{center}

\begin{abstract}
We study high-dimensional numerical integration in the worst-case
setting. The subject of tractability is concerned with the dependence
of the worst-case integration error on the dimension. Roughly
speaking, an integration problem is tractable if the worst-case error
does not explode exponentially with the dimension. Many classical
problems are known to be intractable. However, sometimes tractability can
be shown. Often such proofs are based on randomly selected integration
nodes. Of course, in applications true random numbers are not
available and hence one mimics them with pseudorandom number generators. This motivates us to propose the use of
pseudorandom vectors as underlying integration nodes in order to achieve tractability. In particular, we consider digital inverse vectors and present two examples of problems, the weighted star discrepancy and integration of H\"older continuous, absolute convergent Fourier- and cosine series, where the proposed method is successful.    
\end{abstract}

\centerline{\begin{minipage}[hc]{150mm}{
}
\end{minipage}}

\section{Introduction}

We study numerical integration of multivariate functions $f$ defined
on the $s$-dimensional unit cube by means of quasi-Monte Carlo (QMC)
rules, that is,
$$I_s(f):=\int_{[0,1]^s} f(\bsx) \rd \bsx \approx
\frac{1}{N}\sum_{n=0}^{N-1} f(\bsx_n)=:Q_N(f),$$ 
where we assume that $\bsx_0,\ldots,\bsx_{N-1}$ are well chosen
integration nodes from the unit cube and where $Q_N$ is the QMC-rule based on these nodes. General references for QMC integration are \cite{DKS_acta,DP10,LP14,ni92}.

Usually one studies integrands
from a given Banach space $(\mathcal{F},\|\cdot \|_{\mathcal{F}})$ of
functions. 
As quality criterion we consider the {\it worst-case
error} $$e(Q_N,\mathcal{F})=\sup_{f \in \mathcal{F}\atop
  \|f\|_{\mathcal{F}} \le 1} |I_s(f)-Q_N(f)|.$$ 

For many function classes the problem of QMC integration is
well-studied and one can achieve optimal asymptotic convergence rates
for the worst-case error which are often of the form
$O(N^{-1+\varepsilon})$ for $\varepsilon>0$, with some implicit
dependence on the dimension $s$. Although this is 
the best possible convergence rate in $N$, the dependence on the dimension can be crucial if $s$ is large. This question is the subject of tractability. Tractability means
that we control the dependence of the worst-case error on the
dimension. 

\paragraph{Tractability.} For the numerical integration problem in
$\mathcal{F}$ the {\it $N$th minimal worst-case error} is defined
as $$e(N,s)=\inf_{A_{N,s}} \sup_{f \in \mathcal{F}\atop
  \|f\|_{\mathcal{F}} \le 1} |I_s(f)-A_{N,s}(f)|,$$ where the infimum
is extended over all integration rules (not necessarily QMC rules)
which are based on $N$ function evaluations $f(\bsx_n)$,
$n=0,1,\ldots,N-1$. For $\varepsilon \in (0,1)$ the {\it information
complexity} $N(\varepsilon,s)$
is then defined as the minimal number of function
evaluations which are required in order to achieve a worst-case error
of at most $\varepsilon$. In other words\footnote{Here we speak about the absolute error
criterion. Sometimes one uses the so-called initial error
$e(0,s)=\|I_s\|$ as a reference value and asks for the minimal $N$ in
order to reduce this minimal error by a factor of $\varepsilon$. In
this case one speaks about the normalized error criterion.},  
$$N(\varepsilon,s)=\min\{N \in \mathbb{N}\ :
\ e(N,s) \le \varepsilon\}.$$ 

We say that we have:
\begin{itemize}
\item The {\it curse of dimensionality} if there exist positive
  numbers $C$, $\tau$ and $\varepsilon_0$ such that $$N(\varepsilon,s)
  \ge C (1+\tau)^s\ \ \mbox{ for all } \varepsilon \le \varepsilon_0 \
  \mbox{ and infinitely many } s \in \mathbb{N}.$$ 
\item {\it Polynomial tractability} if there exist non-negative
  numbers $C,\tau_1,\tau_2$ such that $$N(\varepsilon,s) \le C
  s^{\tau_1} \varepsilon^{-\tau_2}\ \ \mbox{ for all } s \in \NN,\
  \varepsilon \in (0,1).$$ 
\item {\it Strong polynomial tractability} if there exist non-negative
  numbers $C$ and $\tau$ such that $$N(\varepsilon,s) \le C
  \varepsilon^{-\tau}\ \ \mbox{ for all } s \in \NN,\  \varepsilon \in
  (0,1).$$ The {\it exponent $\tau^\ast$ of strong polynomial
    tractability} is defined as the infimum of $\tau$ for which strong
  polynomial tractability holds. 
\end{itemize}

Besides polynomial and strong polynomial tractability there are many
other notions of tractability such as {\it weak tractability},
which means that  
$$\lim_{\varepsilon^{-1}+s \rightarrow \infty}
\frac{\log N(\varepsilon,s)}{\varepsilon^{-1}+s}=0.$$ 
Problems that are
not weakly tractable (that is, the information complexity depends
exponentially on $\varepsilon^{-1}$ or $s$) are said to be {\it
  intractable}. The current state of the art in tractability theory is
very well summarized in the three volumes of Novak and Wo\'{z}niakowski
\cite{NW08,NW10,NW12}.

It is known that many multivariate problems defined over standard
spaces of functions suffer from the curse of dimensionality, 
as for example, integration of Lipschitz functions, monotone functions, 
convex functions (see \cite{HNW11}), or smooth functions (see \cite{HNUWI,HNUWII}). The reason for this
disadvantageous behavior may be found in the fact that for standard
spaces all variables and groups of variables are equally important. As
a way out, Sloan and Wo\'{z}niakowski~\cite{SW} suggested to consider
weighted spaces, in which the importance of successive variables and
groups of variables is monitored by corresponding weights, to vanquish
the curse of dimensionality and to obtain polynomial or even strong
polynomial tractability, depending on the decay of the weights.  

It is possible to construct QMC rules for some cases which achieve the one or
other notion of tractability, for example (polynomial) lattice rules for
integration in weighted Korobov spaces or Sobolev spaces. In such
cases the point sets heavily depend on the chosen weights and can 
generally not be used for other weights. A further disadvantage is,
that there is in general no explicit construction for point sets which
can achieve tractability error bounds and thus one relies on computer
search algorithms (for example, the fast component-by-component
constructions, see \cite{DKS_acta,LP14} and the references therein).

On the other hand, for those instances of problems which are tractable
this property is often proved with randomly selected point sets. A
particular example is the ``inverse of star discrepancy problem'' for
which Heinrich, Novak, Wasilkowski, and Wo\'{z}niakowski~\cite{HNWW}
showed with the help of random point sets that the star discrepancy is
polynomially tractable (see also \cite{Aist}). In \cite{AH14} it is even shown that with very high probability (say $99\%$), a randomly selected point set satisfies the aforementioned bounds. However, if we generate a random point set on a computer using a pseudorandom number generator, this result does not apply since the pseudo-random numbers are deterministically constructed. Thus a fundamental question is whether known pseudorandom generators can be used to generate point sets which satisfy discrepancy bounds which imply polynomial tractability.

Thus in this paper we consider point sets
generated by a certain pseudorandom number generator as candidates for
point sets which achieve tractability for certain 
problems. First
attempts in this direction 
have been
made in \cite{d14,dipi14} where
so-called $p$-sets are used (see also the survey \cite{dipiSETA}).

In the present paper we consider another choice of pseudorandom numbers
obtained from explicit inversive pseudorandom number generators. We
show that such point sets can be used to achieve tractability for two
problems, namely the weighted star discrepancy problem
(Section~\ref{sec_wdisc}) and integration of functions from a subclass
of the Wiener algebra which has some additional smoothness properties
(Section~\ref{sec_integ}).   

In the subsequent section we introduce the proposed pseudorandom vectors and prove an estimate on an exponential sum from which we can derive discrepancy- and worst-case error bounds.

\section{Explicit inversive vectors}

Let $\F_q$ be the finite field of order $q=p^k$ with a prime $p$ 
and an integer
$k\geq 1$. Further let $\{\beta_1,\ldots,\beta_k\}$ be an ordered basis of 
$\F_q$ over $\F_p$. 

From a finite vector set in $\F_q^s$
$$\{\bsz_n=(z_{n,1},z_{n,2},\ldots,z_{n,s})\in \F_q^s: n=0,1,\ldots,N-1\},$$ 
we can derive a point set in the $s$-dimensional unit interval. 
More precisely, if 
\begin{equation}
\label{z}
  \bsz_n=\bsc_n^{(1)}\beta_1+\bsc_n^{(2)}\beta_2+\cdots
	+\bsc_n^{(k)}\beta_k
\end{equation}
with all
$\bsc_n^{(j)}=(c_{n,1}^{(j)},c_{n,2}^{(j)},\ldots,c_{n,s}^{(j)})\in
\F_p^s$, then we define an $s$-dimensional {\em digital point set} 
\begin{equation}\label{x}
{\cal P}_s=\left\{\bsx_n=\sum_{j=1}^k \bsc_n^{(j)} p^{-j}\in [0,1)^s :
  n=0,1,\ldots,N-1\right\}. 
\end{equation}

The following point set was essentially introduced in \cite{niwi00}.

\begin{definition}
[Set of explicit inversive points of size $q$]
\label{def1}\rm
Put
$$\overline{z}=\left\{\begin{array}{ll} z^{-1} & \mbox{ if } z\in \F_q^*,\\ 0 & \mbox{ if } z=0.\end{array}\right.$$
For $1\le s\le q$ we choose a subset $S\subseteq\F_q$ of cardinality $s$. 
We consider the vector set 
\begin{equation}\label{Defz}
{\cal S}=\{\bsz_0,\ldots,\bsz_{q-1}\}=\{(\overline{u+v})_{v\in S} : u\in \F_q\}\subset \F_q^s
\end{equation}
of size $q$ and derive ${\cal P}_s=\{\bsx_0,\ldots,\bsx_{q-1}\}\in [0,1)^s$ from ${\cal S}$ by $\eqref{z}$ and $\eqref{x}$. Note that here $N=|\cP_s|=q$.
\end{definition}

Our second point set was essentially introduced in~\cite{meidl2004linear} and is 
defined as follows. 

\begin{definition}[Set of explicit inversive points of period $T$]\label{def2}\rm
Let $0\ne\theta\in\F_q$ be an element of 
multiplicative order $T$ (hence $T| (q-1)$) and $S\subseteq\F_q$ be of cardinality $1\le s\le q$. Then we define
\begin{equation}
\label{eq:definition_winterhof06}
{\cal S}=\{\bsz_0,\ldots,\bsz_{T-1}\}=\{(\overline{\theta^n+v})_{v\in
  S} : n =0,\ldots, T-1\}\subset \F_q^s
\end{equation}
of size $T$ and derive ${\cal P}_s=\{\bsx_0,\ldots,\bsx_{T-1}\}\in
[0,1)^s$ from ${\cal S}$ by $\eqref{z}$ and $\eqref{x}$. 
We remark that in this case  $N=|\cP_s|=T$ and $T$ divides $q-1$.
\end{definition}
We introduce some notation.
\begin{itemize}
\item For $s \in \NN$ put $[s]:=\{1,2,\ldots,s\}$. 
\item For a vector $\bsx=(x_1,x_2,\ldots,x_s)$ and for a nonempty $\mathfrak{u}
  \subseteq [s]$, let $\bsx_{\uu}$ be the projection of $\bsx$ onto
  the components whose index belongs to $\uu$, that is, for
  $\uu=\{u_1,u_2,\ldots,u_w\}$ with $u_1<u_2<\ldots<u_w$ we have
  $\bsx_{\uu}=(x_{u_1},x_{u_2},\ldots,x_{u_w})$. 
\item Let $\psi$ denote the canonical additive character of $\F_q$.  
\item For vectors $\bsx,\bsy \in \F_q^s$ let $\bsx \cdot \bsy \in \F_q$ denote their standard inner product.
\end{itemize}
Now, we are ready to state the first character sum bound.
\begin{lemma}\label{le0}
Let ${\cal S}=\{\bsz_0,\ldots,\bsz_{q-1}\}$ be given by \eqref{Defz} and let
$\emptyset \not=\uu \subseteq [s]$. Then we have $$\max_{\bsw\in
  \F_q^{|\mathfrak{u}|}\setminus\{\bf 0\}} \left|\sum_{n=0}^{q-1}
  \psi(\bsw\cdot \bsz_{n,\mathfrak{u}}) \right|\le
(2|\mathfrak{u}|-2)q^{1/2}+|\mathfrak{u}|+1.$$  
\end{lemma}
\begin{proof}
 Note that the sums to be estimated are of the form
 $$S_q:=\sum_{n=0}^{q-1}\psi(\bsw \cdot \bsz_{n,\mathfrak{u}})=\sum_{u\in \F_q}\psi\left(\sum_{i=1}^{|\uu|} w_i\, \overline{u+v_i}\right)$$
 for some $(w_1,\ldots,w_{|\mathfrak{u}|})\in \F_q^{|\mathfrak{u}|}\setminus\{\bf 0\}$ and $(v_1,\ldots,v_{|\mathfrak{u}|})\in \F_q^{|\mathfrak{u}|}$ with pairwise distinct coordinates 
 $v_i\ne v_j$ if $i\ne j$. 
 
 We proceed as in the proof of \cite[Theorem 1]{niwi00}.
  We have
 $$|S_q|\le |\mathfrak{u}|+\left|\sum_{u\in \F_q, g(u)\ne 0} \psi\left(\frac{f(u)}{g(u)}\right)\right|,$$
 where 
 $$f(x)=\sum_{i=1}^{|\mathfrak{u}|} w_i \prod_{j=1 \atop j\ne i}^{|\mathfrak{u}|} (x+v_j)$$
 and 
 $$g(x)=\prod_{j=1}^{|\mathfrak{u}|} (x+v_j).$$
 Since at least one of the $w_i$ is non-zero and the $v_i$ are distinct, we have $f(-v_i)=w_i \prod_{j\ne i} (v_j-v_i)\ne 0$ and $f$ is not the zero polynomial. 
 Since $\deg(f)<\deg(g)$, by \cite[Lemma 2]{niwi00} the rational function $\frac{f(x)}{g(x)}$ is not of the form $A^p-A$ with some rational function over the algebraic closure of $\F_q$. 
 Hence, we can apply a bound of Moreno and Moreno \cite[Theorem 2]{momo} (see also \cite[Lemma 1]{niwi00}) and the result follows. 
\end{proof}

For the second point set, which was essentially studied
in~\cite{chen2008,wi06},
we also give an analogous character sum bound.
\begin{lemma}
\label{le0a}
  Let $\{\bsz_0,\ldots,\bsz_{T-1}\}$ be given by
  ~\eqref{eq:definition_winterhof06} of size $T$ and let
$\emptyset \not=\uu \subseteq [s]$. Then we have
  \begin{equation*}
    \max_{\bsw\in \F_q^{|\mathfrak{u}|}\setminus\{\bf 0\}}
    \left|\sum_{n=0}^{T-1} \psi(\bsw\cdot \bsz_{n,\mathfrak{u}})
    \right|\le 2|\mathfrak{u}|q^{1/2}+|\mathfrak{u}|. 
\end{equation*}
\end{lemma}
\begin{proof}
The proof is analogous to the proof of ~\cite[Theorem 1]{wi06}.
\end{proof}

\section{The weighted star discrepancy}\label{sec_wdisc}

For an $N$-element point set ${\cal P}_s$ in $[0,1)^s$ the {\it local
discrepancy} $\Delta_{{\cal P}_s}$ is defined
as $$\Delta_{\cP_s}(\bsalpha)=\frac{|{\cal P}_s\cap
  [\bszero,\bsalpha)|}{N}-{\rm Volume}([\bszero,\bsalpha))$$ 
for $\bsalpha=(\alpha_1,\ldots,\alpha_s) \in [0,1]^s$. The {\it star
discrepancy} is then the $L_{\infty}$-norm of the local
discrepancy, $$D_N^{\ast}({\cal P}_s)=\|\Delta_{\cP_s}\|_{L_{\infty}}.$$ 

We consider the weighted star discrepancy. The study of weighted
discrepancy has been 
initiated by Sloan and Wo\'{z}niakowski~\cite{SW} in 1998 in order
to overcome the curse of dimensionality.
Their basic idea was to introduce a set of weights
$\bsgamma=\{\gamma_{\uu}\ : \ \emptyset\ne {\mathfrak u}\subseteq
[s]\}$ which consists of  
non-negative real numbers $\gamma_\uu$. A simple choice of weights are
so-called product weights $(\gamma_j)_{j \ge 1}$, where
$\gamma_\uu = \prod_{j\in \uu} \gamma_j$.
In this case, the weight $\gamma_j$ is associated with the variable $x_j$.

\begin{definition}[Weighted star discrepancy]\rm
For given weights $\bsgamma$ and for a point set $\cP_s$ in $[0,1]^s$
the {\em weighted star discrepancy} is defined as  
$$D_{N,{\bsgamma}}^*({\cal P}_s)=
\max_{\emptyset\ne {\mathfrak u}\subseteq [s]} \gamma_{\mathfrak
  u}\sup_{\bsalpha\in [0,1]^s}|\Delta_{\cP_s}((\bsalpha_{\mathfrak u},\bsone))|,$$ 
where for $\bsalpha=(\alpha_1,\ldots,\alpha_s) \in [0,1]^s$ and for
$\uu \subseteq [s]$ we put $(\bsalpha_{\uu},\bsone)=(y_1,\ldots,y_s)$
with  
$$y_j=\left\{ 
\begin{array}{ll}
\alpha_j & \mbox{ if } j \in \uu,\\
1 & \mbox{ if } j \not\in \uu. 
\end{array}\right.$$
\end{definition}
 
\begin{remark}\rm
Let $\mathcal{F}_1$ be the space of functions with finite
norm $$\|f\|_{\mathcal{F}_1}:=\sum_{\uu \subseteq [s]}
\frac{1}{\gamma_{\uu}} \int_{[0,1]^{|\uu|}}
\left|\frac{\partial^{|\uu|}f}{\partial
    \bsx_{\uu}}(\bsx_\uu,\bsone)\right| \rd \bsx_{\uu},$$ where for $\uu=\emptyset$ we put $\frac{\partial^{|\uu|}f}{\partial
    \bsx_{\uu}}(\bsx_\uu,\bsone)=f(1,1,\ldots,1) $. Then it was
shown in \cite[p.~12]{SW} that the weighted star discrepancy of a point set
$\cP_s$ is an upper bound for the worst-case error of the QMC rule $Q_N$ based on
$\cP_s$, that is, $$e(Q_N,\mathcal{F}_1) \le D_{N,{\bsgamma}}^*({\cal P}_s).$$
\end{remark}

It is well known that there is a close connection between discrepancy
and character sums. In discrepancy theory such relations are known
under the name {\it ``Erd\H{o}s-Tur\'{a}n-Koksma inequalities''}. One
particular instance of an Erd\H{o}s-Tur\'{a}n-Koksma inequality is
given in the following lemma which is perfectly suited for our
applications. Before stating the result, we introduce the following
auxiliary function.  For $q= p^k$, we define
\begin{equation}
  \label{eq:auxiliar}
  \cT(q,s) =
  \begin{cases}
    \left(\frac{k}{2}+1\right)^s & \mbox{if }p=2,\\ 
    \left(\frac{2}{\pi}\log p+\frac{7}{5}\right)^sk^s & \mbox{if }
                                           p>2.\\
  \end{cases}
\end{equation}
The result is the following:
\begin{lemma}\label{le1}
For $q= p^k$ and $\bsz_0,\ldots,\bsz_{N-1}\in \F_q^s$, let ${\cal
  P}_s=\{\bsx_0,\ldots,\bsx_{N-1}\}$ be the $N$-element point set
defined by $\eqref{z}$ and $\eqref{x}$. 
 Then we have 
 $$D_N^*({\cal P}_s)\le  \frac{s}{q}+\frac{\cT(q,s)}{N} \max_{\bsw\in
   \F_q^s\setminus \{\bf 0\}}\left|\sum_{n=0}^{N-1} \psi(\bsw\cdot
   \bsz_n) \right|,
 $$ 
 where $\cT(q,s)$ is defined as in~\eqref{eq:auxiliar}.
\end{lemma}
\begin{proof} 
For a non-zero $s\times k$ matrix $H=(h_{ij})$ with entries $h_{ij}\in
(-p/2,p/2)\cap {\mathbb Z}$ we define the exponential sum 
$$S_N(H)=\sum_{n=0}^{N-1} \exp\left(\frac{2\pi \icomp}{p} \sum_{i=1}^s
  \sum_{j=1}^k h_{ij} c_{n,i}^{(j)}\right),$$ 
where the $c_{n,i}^{(j)}\in \F_p$ are defined by $\eqref{z}$ and where
$\icomp=\sqrt{-1}$. 
By a general discrepancy bound taken from \cite[Theorem 3.12 and Lemma
3.13]{ni92} we get 
$$D_N^*({\cal P}_s)\le \frac{s}{q}+\frac{\cT(q,s)}{N}\max_{H\ne 0}|S_N(H)|,$$
where the maximum is extended over all non-zero matrices $H$  with
entries $h_{ij}\in (-p/2,p/2)\cap {\mathbb Z}$. 

Let $\{\delta_1,\ldots,\delta_k\}$ be the dual basis of the given
ordered basis $\{\beta_1,\ldots,\beta_k\}$ of $\F_q$ over $\F_p$, that
is,
\begin{equation*}
  {\rm Tr}(\delta_j\beta_i)=
  \begin{cases}
    1 &\mbox{if } i=j,\\ 
    0 &\mbox{if } i\ne j,
  \end{cases}
\end{equation*}
where ${\rm Tr}$ denotes the trace function from $\F_q$ to $\F_p$. 
For any basis, there exists a dual basis and this basis is unique, see
\cite[p.\ 55]{lini94} for a proof. Then we have
$$c_{n,i}^{(j)}={\rm Tr}(\delta_jz_{n,i}) \quad \mbox{for } 1\le j\le k, ~1\le i\le s, \mbox{ and } n\in \NN_0,$$
where $\bsz_n=(z_{n,1},\ldots,z_{n,s})$. Therefore
\begin{eqnarray*}
S_N(H)&=&\sum_{n=0}^{N-1} \exp\left(\frac{2\pi \icomp}{p} \ \sum_{i=1}^s \sum_{j=1}^k h_{ij}{\rm Tr}(\delta_j z_{n,i})\right)\\
&=&\sum_{n=0}^{N-1} \exp\left(\frac{2\pi \icomp}{p} \ {\rm Tr}\left(\sum_{i=1}^s \sum_{j=1}^k h_{ij}\delta_j z_{n,i}\right)\right)\\
&=&\sum_{n=0}^{N-1} \psi\left(\sum_{i=1}^s \sum_{j=1}^k h_{ij}\delta_j z_{n,i}\right).
\end{eqnarray*}
Put
$$\bsw=(w_1,\ldots,w_s)\in \F_q^s\mbox{ with } w_i=\sum_{j=1}^k h_{ij}\delta_j\mbox{ for }i=1,\ldots,s.$$
Since $H$ is not the zero matrix and $\{\delta_1,\ldots,\delta_k\}$ is
a basis of $\F_q$ over $\F_p$, it follows that $\bsw$ is not the zero
vector. This fact finishes the proof.
\end{proof}
Now we extend the star discrepancy estimate from Lemma~\ref{le1} to
the weighted star discrepancy: 

\begin{lemma}\label{le3}
For $\bsz_0,\ldots,\bsz_{N-1}\in \F_q^s$ let ${\cal
  P}_s=\{\bsx_0,\ldots,\bsx_{N-1}\}$ be the $N$-element point set
defined by $\eqref{z}$ and $\eqref{x}$. 
Then we have 
\begin{equation*}
D_{N,\bsgamma}^*({\cal P}_s)\le  \max_{\emptyset \ne
  \mathfrak{u}\subseteq [s]} \gamma_{\mathfrak{u}} \left (  \frac{|\mathfrak{u}|}{q}
 + \frac{\cT(q,|\mathfrak{u}|)}{N}
 \max_{\bsw\in \F_q^{|\mathfrak{u}|}\setminus\{\bf 0\}} \left|\sum_{n=0}^{N-1} \psi(\bsw\cdot \bsz_{n,\mathfrak{u}}) \right|\right).
\end{equation*}
where $\bsz_{n,\mathfrak{u}}\in \F_q^{|\mathfrak{u}|}$ is the projection of $\bsz_n$ to the coordinates indexed by 
 $\mathfrak{u}$.
\end{lemma}
     
\begin{proof}
The result follows immediately from Lemma~\ref{le1} together with the
fact that $$D_{N,\bsgamma}^{\ast}(\cP_s) \le \max_{\emptyset \not=\uu
  \subseteq [s]} \gamma_{\uu} D_N^{\ast}(\cP_\uu),$$ where $\cP_\uu$
consists of the points from $\cP_s$ projected onto the components whose
indices belong to $\uu$.  
\end{proof}
The previous lemma gives us our first main result.
\begin{theorem}\label{thm1}
For the point set $\cP_s$ defined as in Definition~\ref{def1} the following bound holds:
\begin{equation*}
  D_{q,\bsgamma}^*({\cal P}_s)\le \max_{\emptyset \ne
  \mathfrak{u}\subseteq [s]} \gamma_{\mathfrak{u}}|\mathfrak{u}|
\left ( \frac{1}{q} + \frac{3\cT(q,|\mathfrak{u}|)
}{q^{1/2}} \right ). 
\end{equation*}
For the point set $\cP_s$ defined as in Definition~\ref{def2} the following bound holds:
\begin{equation*}
  D_{T,\bsgamma}^*({\cal P}_s)\le \max_{\emptyset \ne
  \mathfrak{u}\subseteq [s]} \gamma_{\mathfrak{u}}|\mathfrak{u}|
\left ( \frac{1}{q} + \frac{3\cT(q,|\mathfrak{u}|)
q^{1/2}}{T} \right ). 
\end{equation*}
\end{theorem}
\begin{proof}
The result follows from Lemma~\ref{le3} and Lemmas~\ref{le0} and~\ref{le0a}.
\end{proof}

We point out that the discrepancy estimate from Theorem~\ref{thm1}
holds for every choice of weights. Also, it is important to remark
that point sets defined by~\eqref{eq:definition_winterhof06} are more
flexible in terms of size. It is easy to check that for any prime $p$
and $T$ not divisible by $p$, there exists $q=p^k$, such that $T$
divides $q-1$. This $k$ is the multiplicative order of $p$ modulo
$T$. Since the multiplicative group of $\F_q$ is cyclic there is an element $\theta\in \F_q$ of order $T$.

\paragraph{Tractability.} For a recent overview of results concerning
tractability properties of the weighted star discrepancy we refer to
\cite{dipiSETA,dipi14}.  
\begin{iremark}\rm
Unfortunately 
the proof of \cite[Theorem~3.2
(ii)]{dipi14} (also \cite[Theorem~7(2)]{dipiSETA}) is not correct and
hence this part of the theorem must be discarded. All other parts of these papers are correct.
\end{iremark}

We now restrict ourselves to product weights and present a condition
on the weights for strong polynomial tractability.  

\begin{theorem}\label{thm2}
Let $\cP_s$ be the point set from Definition~\ref{def1}.
Assume that for an ordered sequence of weights $\bsgamma=(\gamma_j)_{j
  \ge 1}$ with $\gamma_1 \ge \gamma_2 \ge \ldots$, there is a $0 \le
\delta < 1/2$ such that 
\begin{equation}\label{bed1}
\limsup\limits_{j\to \infty} j \gamma_j < \frac{\delta}{3}.
\end{equation}
Then there is a constant $c_{\bsgamma, \delta} > 0$, which depends
only on $\bsgamma$ and $\delta$ but not on $s$ such that for all $1\le s\le q$ we have
\begin{equation*}
D^*_{q,\bsgamma}({\cal P}_s) \le \frac{c_{\bsgamma, \delta}}{q^{1/2 - \delta}}.
\end{equation*}
If $\limsup_{j\to \infty} j \gamma_j = 0$, then the result holds for all $\delta > 0$.
\end{theorem}

\begin{proof}
We show the result for $p>2$ and when the weights satisfy~\eqref{bed1}
only. The other cases are proven in a similar way.
From Theorem~\ref{thm1}, the fact
that $r\le 2^r$ for $r\ge 1$ and using
the ordering of $\bsgamma$ we have 
\begin{eqnarray*}
D^*_{q, \bsgamma}({\cal P}_s) &\le& 
C^{(1)}\max_{r=1,\ldots, s}\frac{r\prod_{j=1}^{r}\left (
    \gamma_j k \left(\frac{2}{\pi} \log p + \frac{7}{5}  \right)
  \right)}{q^{1/2}}\\
&\le& C^{(1)} \max_{r=1,\ldots, s}\frac{\prod_{j=1}^{r}\left (
    2\gamma_j k \left(\frac{2}{\pi} \log p + \frac{7}{5}  \right)
  \right)}{q^{1/2}}.
\end{eqnarray*}
Notice that $ 2 k (\tfrac{2}{\pi} \log p + \tfrac{7}{5}  )$ can also be bounded by $c\log q$ for some $0 < c<3$. 
Now, let $\ell$ be the largest integer such that $
c\gamma_{\ell} \log q > 1$. Then we have
\begin{equation*}
D^*_{q, \bsgamma}({\cal P}_s) \le C^{(1)} \frac{\prod_{j=1}^\ell (c \gamma_j \log q)}{q^{1/2}}.
\end{equation*}

The condition $\limsup_{j\to \infty} j \gamma_j < \delta/3$ implies
that there is an $L > 0$ such that $j \gamma_j  < \delta/3$ for all $j
\ge L$.  Without loss of generality we may assume that $\ell \ge
L$. (Otherwise, if $\ell <L$, consider a new weight sequence
$\bsgamma'=(\gamma_j')_{j \ge 1}$ with $\gamma'_j=\gamma_j$ for all $j
\in \{1,\ldots,\ell\} \cup \{L,L+1,\ldots\}$ and
$\gamma_j'=\gamma_{\ell}$ for $j \in \{\ell+1,\ldots,L-1\}$, and hence
$\gamma_j \le \gamma_{j}'$ for all $j \ge 1$).

For $r \in \NN$ let
\begin{equation*}
c_r = \prod_{j=1}^r (c \gamma_j \log q),
\end{equation*}
so we have 
\begin{equation*}
\frac{c_{\ell}}{c_{\ell - 1}} = c \gamma_{\ell} \log q < \frac{c
  \delta}{3 \ell} \log q. 
\end{equation*}
By the definition of $\ell$ we have $c_{\ell - 1} < c_{\ell}$,
hence $$1 < \frac{c\delta}{3 \ell} \log q,$$ which implies $\ell <  c\delta
(\log q ) / 3$, or $\ell \le \lfloor c \delta (\log q) /3 \rfloor$. 

Therefore, there is a constant $C_{\bsgamma}^{(2)} > 0 $ such that
\begin{align*}
\prod_{j=1}^\ell (c \gamma_j \log q) = & \prod_{j=1}^{L-1} (c \gamma_j
                                         \log q) \prod_{j=L}^\ell (c
                                         \gamma_j \log q)\\ 
\le & C_{\bsgamma}^{(2)} (c\log q)^{L-1} \prod_{j=L}^{\lfloor c \delta
      (\log q ) / 3 \rfloor} \frac{c \delta \log q}{3 j}. 
\end{align*}
Let $x:=c \delta (\log q ) / 3$.
Then 
\begin{align*}
\prod_{j=1}^\ell (c \gamma_j \log q) \le & C_{\bsgamma}^{(2)}
                                           \left(\frac{c\log
                                           q}{x}\right)^{L-1} ((L-1)!)
                                           \ \frac{x^{\lfloor x
                                           \rfloor}}{\lfloor x
                                           \rfloor!}\\ 
\le & C_{\bsgamma}^{(2)} \left(\frac{3}{ \delta}\right)^{L-1}
      ((L-1)!)\ {\rm e}^x\\ 
= & C_{\bsgamma,\delta}^{(3)} \ q^{c \delta/3}\\
< & C_{\bsgamma,\delta}^{(3)} \ q^{\delta},
\end{align*}
where $C_{\bsgamma,\delta}^{(3)}=C_{\bsgamma}^{(2)} \left(\tfrac{3}{
    \delta}\right)^{L-1}  ((L-1)!)$ (note that $L$ only depends on
$\bsgamma$). This implies
$$
D^*_{q, \bsgamma}({\cal P}_s) \le C^{(1)} \frac{3 C_{\bsgamma,\delta}^{(3)}
}{q^{1/2-\delta}}$$ 
and finishes the proof. 
\end{proof}

\begin{remark}\rm
Note that $\sum_{j=1}^{\infty} \gamma_j < \infty$, together with the
monotonicity $\gamma_1 \ge \gamma_2 \ge \ldots$ implies $\limsup_{j\to
  \infty} j \gamma_j =0$.  
To see this let $\varepsilon>0$. From $\sum_{j=1}^{\infty} \gamma_j <
\infty$ it follows with the Cauchy condensation test that also
$\sum_{k=0}^{\infty} 2^k \gamma_{2^k}< \infty$. In particular, $2^k
\gamma_{2^k} \rightarrow 0$ for $k \rightarrow \infty$. This means that
$\gamma_{2^k} \le \varepsilon/2^{k+1}$ for $k$ large enough. Thus, for
large enough $j$ with $2^k \le j < 2^{k+1}$ we obtain $$\gamma_j \le
\gamma_{2^k} \le \frac{\varepsilon}{2^{k+1}} <
\frac{\varepsilon}{j}.$$ In particular, for $j$ large enough we have
$j \gamma_j < \varepsilon$. This implies that $$\limsup_{j\to \infty}
j \gamma_j =0.$$  

Of course the converse is not true in general (for example, $\gamma_j=1/(j \log j)$).
\end{remark}

\begin{corollary}\label{cor1}
With the notation and conditions as in Theorem~\ref{thm2}, in
particular $\limsup_{j \rightarrow \infty}  \gamma_j < \delta/3$, the
weighted star discrepancy (or, equivalently, integration in
$\mathcal{F}_1$) is strongly polynomially tractable with
$\varepsilon$-exponent at most  $2/(1-2 \delta)$. 
\end{corollary}

\begin{proof}
For $\varepsilon>0$ let $M:=\lceil (c_{\bsgamma,\delta}
\varepsilon^{-1})^{2/(1-2\delta)}\rceil$. Let $q$ be the smallest
prime power which is greater or equal to $M$. According to the
Postulate of Bertrand we have $q < 2M$. Then we have
$D_{q,\bsgamma}^{\ast}(\cP_s) \le \varepsilon$ and hence the information
complexity satisfies $$N(\varepsilon,s) \le q \le 2 M = 2 \lceil
(c_{\bsgamma,\delta} \varepsilon^{-1})^{2/(1-2\delta)}\rceil.$$ This means that we have strong polynomial tractability.
\end{proof} 

For the proof of Corollary~\ref{cor1} it is enough to use the construction of Definition~\ref{def1} with a prime $q$.
However, from a practical point of view the construction of Definition~\ref{def1} with any prime power $q$ and the construction of Definition~\ref{def2} provide more flexibility.
In particular the case $q=2^r$ can be efficiently implemented using (optimal) normal bases and the Itoh-Tsujii inversion algorithm, see \cite[Chapter 3]{J} and \cite{IT}, respectively.

\section{Integration of H\"older continuous, absolutely convergent Fourier series and cosine series}\label{sec_integ} 

\paragraph{Absolutely convergent Fourier series.} For $f \in L_2([0,1]^s)$ and $\bsh \in \ZZ^s$ 
we define 
the $\bsh$th Fourier coefficient of $f$ as $\widehat{f}(\bsh)=\int_{[0,1]^s} f(\bsx) {\rm e}^{-2 \pi \icomp \bsh \cdot \bsx} \rd \bsx$. Then we can associate to $f$ its Fourier series 
\begin{equation}\label{fser}
f(\bsx) \sim \sum_{\bsh \in \ZZ^s} \widehat{f}(\bsh) {\rm e}^{2 \pi \icomp \bsh \cdot \bsx}.
\end{equation}

Let $\alpha \in (0,1]$ and $t\in [1,\infty]$. Similarly to \cite{d14} we consider the norm 
\begin{equation*} 
\|f\|_{K_{\alpha,t}} = \sum_{\uu \subseteq [s] } |\uu|
\sum_{\boldsymbol{k}_{\uu} \in \mathbb{Z}_{\ast}^{|\uu|} } |
\widehat{f}(\boldsymbol{k}_{\uu}, \boldsymbol{0} )| + |f|_{H_{\alpha,
    t}}, 
\end{equation*}
where $\ZZ_{\ast}=\ZZ\setminus\{0\}$ and where
\begin{equation*}
|f|_{H_{\alpha, t}} = \sup_{\boldsymbol{x}, \boldsymbol{x} + \boldsymbol{h} \in [0,1]^s} \frac{|f(\boldsymbol{x} + \boldsymbol{h}) - f(\boldsymbol{x}) |}{\|\boldsymbol{h} \|_{\ell_t}^\alpha},
\end{equation*}
is the H\"older semi-norm where $\|\cdot\|_{\ell_t}$ denotes the norm in $\ell_t$. 

We define the following sub-class of the Wiener algebra $$K_{\alpha,t}:=\{f \in L_2([0,1]^s) \ : \ f \mbox{ is one-periodic and } \|f\|_{K_{\alpha,t}} < \infty\}.$$ The choice of $t$ will influence the dependence on the dimension of the worst-case error upper bound. As in \cite{d14} we remark that for any $f \in K_{\alpha,t}$ the Fourier series \eqref{fser} of $f$ converges to
$f$ at every point $\bsx \in [0,1]^s$. This follows directly from \cite[Corollary 1.8, p.\ 249]{stei}, using that $f$ is continuous since it satisfies a H\"older condition, i.e.
$|f|_{H_{\alpha, t}}< \infty$. More information on $K_{\alpha,t}$ can be found in \cite{d14}.

\begin{theorem}
Let $\cP_s$ be the point set from Definition~\ref{def1} with $k = 1$ and $q=p=N$ and let $Q_N$ be the QMC rule based on $\cP_s$. Then for $\alpha \in (0,1]$ and $t\in [1,\infty]$ we have $$e(Q_N,K_{\alpha,t}) \le \max\left( \frac{3}{\sqrt{N}}, \frac{s^{\alpha/t}}{N^\alpha} \right). $$ In particular, if $t=\infty$ we have $$e(Q_N,K_{\alpha,\infty}) \le \frac{3}{N^{\min(\alpha, 1/2)} }. $$
\end{theorem}

\begin{proof}
For $f \in K_{\alpha,t}$ we have
\begin{eqnarray*}
\lefteqn{\left|\frac{1}{N} \sum_{n=0}^{N-1} f(\boldsymbol{x}_n) - \int_{[0,1]^s} f(\boldsymbol{x}) \,\mathrm{d} \boldsymbol{x} \right| = \left| \sum_{\bsk \in \mathbb{Z}^s \setminus \{\bszero\}} \widehat{f}(\bsk) \frac{1}{N} \sum_{n=0}^{N-1} {\rm e}^{2 \pi \icomp \bsk \cdot \bsx_n}\right|} \\  
& \le & \sum_{\bsk \in \mathbb{Z}^s \atop N \nmid \bsk} |\widehat{f}(\bsk)| \frac{1}{N} \left| \sum_{n=0}^{N-1} {\rm e}^{\frac{2 \pi \icomp}{p} \bsk \cdot \bsc_n}\right| +  \sum_{\boldsymbol{k} \in \mathbb{Z}^s \setminus \{\boldsymbol{0} \} \atop  N \mid \boldsymbol{k} } |\widehat{f}(\boldsymbol{k}) |\\
& = & \sum_{\emptyset \not=\uu \subseteq [s]} \sum_{\bsk_{\uu} \in \mathbb{Z}_{\ast}^{|\uu|} \atop N \nmid \bsk_{\uu}} |\widehat{f}((\bsk_{\uu},\bszero))| \frac{1}{N} \left| \sum_{n=0}^{N-1} {\rm e}^{\frac{2 \pi \icomp}{p} \bsk_{\uu} \cdot \bsc_{n,\uu}}\right| +  \sum_{\boldsymbol{k} \in \mathbb{Z}^s \setminus \{\boldsymbol{0} \}} |\widehat{f}(N \boldsymbol{k}) |,
\end{eqnarray*}
where $N\mid \bsk$ if all coordinates of $\bsk$ are divisible by $N$ and $N\nmid \bsk$ otherwise.
Now we apply Lemma~\ref{le0} to the first sum and \cite[Lemma~1]{d14} to the second sum and obtain
\begin{eqnarray*}
\lefteqn{\left|\frac{1}{N} \sum_{n=0}^{N-1} f(\boldsymbol{x}_n) - \int_{[0,1]^s} f(\boldsymbol{x}) \,\mathrm{d} \boldsymbol{x} \right|}\\
& \le & \frac{3}{\sqrt{N}} \sum_{\emptyset \not=\uu \subseteq [s]} |\uu| \sum_{\boldsymbol{k}_{\uu} \in \mathbb{Z}_{\ast}^{|\uu|} \atop N \nmid \boldsymbol{k}_{\uu} }  |\widehat{f}((\boldsymbol{k}_{\uu},\bszero))|  + \frac{s^{\alpha/t}}{N^{\alpha}} |f|_{H_{\alpha,t}}\\ 
& \le &  \max\left( \frac{3}{\sqrt{N}}, \frac{s^{\alpha/t}}{N^\alpha} \right) \|f\|_{K_{\alpha, t}}.
\end{eqnarray*}
The result follows. 
\end{proof}

\begin{corollary}
Integration in $K_{\alpha,\infty}$ is strongly polynomially tractable with $\varepsilon$-exponent at most $\max(\tfrac{1}{\alpha},2)$. 
\end{corollary}

\begin{proof}
The proof is similar to the one of Corollary~\ref{cor1}. 
\end{proof}

\paragraph{Absolutely convergent cosine series.} So far we required that the functions are periodic. Now we show how we can get rid of this assumption. 
Let us consider cosine series instead of classical Fourier series.

The cosine system $\{\cos(k \pi x) \ : \ k \in \NN_0\}$ forms a complete orthogonal basis of $L_2([0,1])$. To get an ONB we need to normalize it. Hence we define
$$\sigma_k(x):=\left\{
\begin{array}{ll}
1 & \mbox{ if } k=0,\\
\sqrt{2} \cos(k \pi x) & \mbox{ if } k \in \NN, 
\end{array}\right.$$
then the system $$\{\sigma_k \ : \ k \in \NN_0\}$$ forms an ONB of $L_2([0,1])$.
For $\bsk=(k_1,\ldots,k_s) \in \NN_0^s$ and $\bsx=(x_1,\ldots,x_s) \in [0,1]^s$ define $$\sigma_{\bsk}(\bsx) =\prod_{j=1}^s \sigma_{k_j}(x_j).$$ The $\sigma_{\bsk}$ with $\bsk \in \NN_0^s$ constitute an ONB of $L_2([0,1]^s)$.

To an $L_2$-function $g$ we can associate the cosine series $$g(\bsx) \sim \sum_{\bsk \in \NN_0^s} \widetilde{g}(\bsk) \sigma_{\bsk}(\bsx)$$ with cosine coefficients $\widetilde{g}(\bsk)=\int_{[0,1]^s} g(\bsx) \sigma_{\bsk}(\bsx) \rd \bsx$.

In order to apply the results for Fourier series to cosine series we need the tent transformation $\phi:[0,1] \rightarrow [0,1]$ given by $$\phi(x)=1-|2 x-1|.$$  For vectors $\bsx$ the tent transformed point $\phi(\bsx)$ is understood component wise. The tent transformation is a Lebesgue measure preserving map and we have $$\int_{[0,1]^s} g(\bsx) \rd \bsx = \int_{[0,1]^s} g(\phi(\bsx)) \rd \bsx.$$

Define the norm $$\|g\|_{C_{\alpha,t}}=\sum_{u \subseteq [s]} |\uu| 2^{|\uu|/2} \sum_{\bsk_{\uu} \in \NN^{|\uu|}} |\widetilde{g}((\bsk_{\uu},\bszero))|+ 2^{\alpha} |g|_{H_{\alpha,t}}$$ and let $$C_{\alpha,t}=\{g \in L_2([0,1]^s) \ : \ \|g\|_{C_{\alpha,t}} < \infty\}.$$ For $g \in L_2([0,1]^s)$ we have that the function $f=g\circ \phi$ is one-periodic and $$\|f\|_{K_{\alpha,t}}=\|g\|_{C_{\alpha,t}}.$$ The cosine series of $g$ converges point-wise and absolute to $g$ for all points in $[0,1]^s$.

Now we consider integration of functions from $C_{\alpha,t}$: For a point set $\cP_s=\{\bsx_0,\ldots,\bsx_{N-1}\}$ let $\cQ=\{\phi(\bsx_0),\ldots,\phi(\bsx_{N-1})\}$ be the tent transformed version of $\cP_s$. Let $Q_N$ be a QMC rule based on $\cP_s$. Then we denote by $Q_N^{\phi}$ the QMC rule based on $\cQ$.

As in \cite[Proof of Theorem~2]{d14} we have the identity of worst-case errors  in $K_{\alpha,t}$ and $C_{\alpha,t}$ when we switch from a QMC rule to the tent transformed version of this rule, namely $$e(Q_N,K_{\alpha,t})=e(Q_N^{\phi},C_{\alpha,t}).$$ With this identity we can transfer the results for periodic functions to not necessarily periodic ones.

\begin{corollary}
Let $\cP_s$ be the point set from Definition~\ref{def1} with $k = 1$ and $q=p=N$ and let $Q_N^{\phi}$ be the tent transformed version of the QMC rule based on $\cP_s$. Then for $\alpha \in (0,1]$ and $t\in [1,\infty]$ we have $$e(Q_N^\phi,C_{\alpha,t}) \le \max\left( \frac{3}{\sqrt{N}}, \frac{s^{\alpha/t}}{N^\alpha} \right). $$ In particular, if 
$t=\infty$ we have $$e(Q_N^\phi,C_{\alpha,\infty}) \le \frac{3}{N^{\min(\alpha, 1/2)} }. $$ 
\end{corollary}

\begin{corollary}
Integration in $C_{\alpha,\infty}$ is strongly polynomially tractable with $\varepsilon$-exponent at most $\max(\tfrac{1}{\alpha},2)$. 
\end{corollary}

\noindent\textbf{Authors' addresses:}\\

\noindent Josef Dick, School of Mathematics and Statistics, The University of New South Wales, Sydney NSW 2052, Australia. Email: \texttt{josef.dick@unsw.edu.au}\\

\noindent Domingo Gomez-Perez, Faculty of Sciences, University of Cantabria, E-39071 Santander, Spain. Email: \texttt{domingo.gomez@unican.es}\\

\noindent Friedrich Pillichshammer, Department for Financial Mathematics and Applied Number Theory, Johannes Kepler University Linz, Altenbergerstr.\ 69, 4040 Linz, Austria. Email: \texttt{friedrich.pillichshammer@jku.at}\\ 

\noindent Arne Winterhof, Johann Radon Institute for Computational and Applied Mathematics, Austrian Academy of Sciences, Altenbergerstr.\ 69, 4040 Linz, Austria. Email: \texttt{arne.winterhof@oeaw.ac.at}

\end{document}